\documentclass[a4paper,12pt]{article}

\usepackage{amssymb,,amsmath}
\usepackage{graphicx}

\newtheorem{theorem}{Theorem}[section]
\newtheorem{proposition}[theorem]{Proposition}

\newtheorem{lemma}[theorem]{Lemma}

\newenvironment{proof}{{\bf Proof }}{\hfill $\Box$}

\newcommand{\CC}{\mathbb{C}}
\newcommand{\NN}{\mathbb{N}}
\newcommand{\RR}{\mathbb{R}}
\newcommand{\EE}{\mathbb{E}}

\newcommand{\ZZ}{\mathbb{Z}}
\newcommand{\PP}{\mathbb{P}}
\def\ps#1#2{\langle\,{#1}\,,\,{#2}\,\rangle}

\newcommand{\rB}{\mathcal{B}}

\newcommand{\rF}{\mathcal{F}}
\newcommand{\rE}{\mathcal{E}}
\newcommand{\rH}{\mathcal{H}}
\newcommand{\rK}{\mathcal{K}}
\newcommand{\rN}{\mathcal{N}}
\newcommand{\rS}{\mathcal{S}}
\newcommand{\rM}{\mathcal{M}}

\newcommand{\rL}{\mathcal{L}}
\newcommand{\rV}{\mathcal{V}}

\newcommand{\Tr}{\mathrm{Tr}}

\newcommand{\s}{\sigma}
\font\timesept=cmr7
\font\timehuit=cmr8

\def\O{\Omega}

\def\tr{{\mbox{tr}}}

\def\O{\Omega}
\def\wt{\widetilde}
\def\indic{{\mathop{\rm 1\mkern-4mu l}}}
\def\qq{\qquad}

\def\wt{\widetilde}

\def\cd{\cdot}

\def\ol{\overline}
\def\sm{{\scriptstyle -}}

\def\Tr{\mathop{\rm Tr\,}\nolimits}
\def\tr{\mathop{\rm Tr\,}\nolimits}
\def\tri{\mathop{\rm Tr}\nolimits}

\def\ker{\mathop{\rm Ker\,}\nolimits}

\def\ps#1#2{\langle #1\, ,\, #2\rangle}

\def\norme#1{\left\| #1\right\|}
\def\normca#1{{\left\| #1\right\|}^2}
\def\ab#1{\left\vert #1\right\vert}

\def\vect#1{\vec{#1}}

\def\O{\Omega}

\def\D{\Delta}
\def\a{\alpha}
\def\b{\beta}
\def\s{\sigma}

\def\r{\rho}

\def\e{\varepsilon}

\def\l{\lambda}
\def\g{\gamma}
\def\o{\omega}

\usepackage{color}
\definecolor{Red}{rgb}{1,0,0}
 
\def\aaa{{\mathcal A}}

\begin{document}


\title{Central Limit Theorems\\for Open Quantum Random Walks\\and Quantum Measurement Records\footnote{Work supported by ANR project ``HAM-MARK", N${}^\circ$ ANR-09-BLAN-0098-01}}

\author{S. Attal, N. Guillotin-Plantard, C. Sabot}

\date{}

\maketitle

\begin{abstract}
Open Quantum Random Walks, as developed in \cite{APSS}, are a quantum generalization of Markov chains on finite graphs or on lattices. These random walks are typically quantum in their behavior, step by step, but they seem to show up a rather classical asymptotic behavior, as opposed to the quantum random walks usually considered in Quantum Information Theory (such as the well-known Hadamard random walk). Typically, in the case of Open Quantum Random Walks on lattices, their distribution seems to always converge to a Gaussian distribution or a mixture of Gaussian distributions. In the case of nearest neighbors homogeneous Open Quantum Random Walks on $\ZZ^d$ we prove such a Central Limit Theorem, in the case where only one Gaussian distribution appears in the limit. Through the quantum trajectory point of view on quantum master equations, we transform the problem into studying a certain functional of a Markov chain on $\ZZ^d$ times the Banach space of quantum states. The main difficulty is that we know nothing about the invariant measures of this Markov chain, even their existence. Surprisingly enough, we are able to produce a Central Limit Theorem with explicit drift and explicit covariance matrix. The interesting point which appears with our construction and result is that it applies actually to a wider setup: it provides a Central Limit Theorem for the sequence of recordings of the quantum trajectories associated to any completely positive map. This is what we show and develop as an application of our result.

In a second step we are able to extend our Central Limit Theorem to the case of several asymptotic Gaussians, in the case where the operator coefficients of the quantum walk are block-diagonal in a common basis.
\end{abstract}

\tableofcontents

\section{Introduction}
Quantum Random Walks, such as the Hadamard quantum random walk, are nowadays a very active subject of investigations, with applications in Quantum Information Theory in particular (see \cite{Kem} for a survey). 
These quantum random walks are particular discrete-time quantum dynamics on a state space of the form $\rH\otimes \CC^{\ZZ^d}$. The space $\CC^{\ZZ^d}$ stands for a state space labelled by a lattice $\ZZ^d$, while the space $\rH$ stands for the degrees of freedom given on each point of the lattice. The quantum evolution concerns pure states of the system which are of the form 
$$
\vert\Psi\rangle=\sum_{i\in\ZZ^d} \vert\varphi_i\rangle\otimes\vert i\rangle\,.
$$
After one step of the dynamics, this state is transformed into another pure state,
$$
\vert\Psi'\rangle=\sum_{i\in\ZZ^d} \vert\varphi'_i\rangle\otimes\vert i\rangle\,.
$$
Each of these two states gives rise to a probability distribution on $\ZZ^d$, the one we would obtain by measuring the position on $\CC^{\ZZ^d}$:
$$
\mbox{Prob}(\{i\})=\normca{\varphi_i}\,.
$$
The point is that the probability distribution associated to $\vert \Psi'\rangle$ cannot be deduced from the distribution associated to $\vert \Psi\rangle$ by ``classical rules", that is, there is no classical probabilistic model (such as a Markov transition kernel, or similar) which gives the distribution of $\vert \Psi'\rangle$ in terms of the one of $\vert\Psi\rangle$. One needs to know the whole state $\vert \Psi\rangle$ in order to compute the distribution of $\vert\Psi'\rangle$.

These quantum random walks, have been successful for they give rise to strange behaviors of the probability distribution as time goes to infinity. In particular one can prove that they satisfy a rather surprising Central Limit Theorem whose speed is $n$, instead of $\sqrt n$ as usually, and the limit distribution is not Gaussian, but more like functions of the form (see \cite{Kon})
$$
x\mapsto\frac{\sqrt{1-a^2}\,(1-\l x)}{\pi\,(1-x^2)\,\sqrt{a^2-x^2}}\,,
$$
where $a$ and $\l$ are constants.

\smallskip
In the article \cite{APSS} is introduced a new family of quantum random walks, called \emph{Open Quantum Random Walks}. These random walks deal with density matrices instead of pure states, that is, on a state space $\rH\otimes \CC^{\ZZ^d}$ they consider density matrices of the form
$$
\r=\sum_{i\in\ZZ^d}\r_i\otimes\vert i\rangle\langle i\vert\,.
$$
To this density matrix is attached a probability distribution, associated to the values one would obtain by measuring the position:
$$
\mbox{Prob}(\{i\})=\tr(\r_i)\,.
$$
After one step of the dynamics, the density matrix evolves to another state of the same form
$$
\r'=\sum_{i\in\ZZ^d}\r'_i\otimes\vert i\rangle\langle i\vert\,,
$$
with the associated new distribution. 

In \cite{APSS} it is proved that these Open Quantum Random Walks are a non-commutative extension of all the classical Markov chains, that is, they contain all the classical Markov chains as particular cases, but they also describe typically quantum behaviors. 

\smallskip
Though, as shown on simulations in the same article, it seems that Open Quantum Random Walks of infinite lattices such as $\ZZ^d$ exhibit a rather classical behavior in the limit, that is, their limit distribution seems to always converge to a Gaussian distribution, or to a mixture of Gaussian distributions (including the case of Dirac masses as particular cases of Gaussian distributions). While the quantum random walk, step by step, seems to be very quantum, that is, the distribution at time $n+1$ has nothing to do with the distribution at time $n$ (at least it cannot be deduced from it 
without the complete information of the full density matrix), it appears that asymptotically the quantum random walks becomes more and more classical. 

\smallskip
The aim of this article is to prove, under some conditions, a Central Limit Theorem for these Open Quantum Random Walks and to compute explicitly the characteristics of the associated Gaussian distribution: drift and covariance matrix. 

\smallskip
This article is structured as follows. In Section \ref{S:notations} we recall a certain number of notations and concepts which are very common in the context of Quantum Mechanics: states, density matrices, completely positive maps, etc. Section \ref{S:OQRW} is then devoted to presenting the general mathematical structure of Open Quantum Random Walks and their probability distributions. We end up this section with a series of examples and numerical simulations which illustrate our definitions and which will be covered later on by our Central Limit Theorems.  In Section \ref{S:traj} we explain  the Quantum Trajectory approach to Quantum Master Equations. This approach, which is nowadays very important in the study of Open Quantum Systems, gives a way for Open Quantum Random Walks to be simulated by means of a particular Banach space-valued classical Markov process. In the same section we recall an important ergodic property of quantum trajectories, as proved in \cite{K-M}.

\smallskip
The last sections are the ones where the main theorems are proved. First of all the main Central Limit Theorem is proved in the context of a single asymptotic Gaussian distribution. The proof is based on proving a Central Limit Theorem for a particular martingale associated to the quantum trajectories. This martingale is obtained by the usual method of solving the Poisson equation, which surprisingly can be implemented explicitly in our context, even though we do not have any information on the existence of an invariant measure for the Markov chain associated to quantum trajectories.  Furthermore the parameters of the limit Gaussian distribution are explicitly obtained. 

We then show how our main theorem applies to a wider context: a Central Limit Theorem for the measurement records of a discrete-time trajectory. 

We finally extend the Central Limit Theorem to a context with several asymptotic Gaussians, but with block-diagonal coefficients for the Open Quantum Random Walk. We prove that, in this case, the Open Quantum Random Walk behaves like a mixture of Open Quantum Random Walks with single Gaussian, that is, up to conditioning the trajectories at the beginning, we get a behavior of an OQRW with a single asymptotic Gaussian. We compute several examples which illustrate the different situations of our theorems, we compute the associated asymptotic parameters.

\section{General Notations}\label{S:notations}
We recall here some useful notations and terminologies that shall be used in this article.

All our Hilbert spaces are on the complex field and are separable (if not finite dimensional). For all Hilbert space $\rH$ we denote by $\rB(\rH)$ the Banach space of bounded operators on $\rH$ equipped with the usual operator-norm that we denote by $\norme{\cdot}_\infty$. We denote by $\rL_1(\rH)$ the Banach space of trace-class operators on $\rH$, equipped with the trace-norm $\norme{\cdot}_1$. 

\smallskip
Let $\rH$ be a Hilbert space. For any $\phi\in\rH$ we put $\vert \phi\rangle$ to simply denote the element $\phi$ of $\rH$ (more rigorously, it should be the operator $\l\mapsto \l\phi$ from $\CC$ to $\rH$). We define
$$
\begin{matrix}
\langle\phi\vert&:&\rH&\longrightarrow&\CC\\
&&x&\longmapsto&\ps{\phi}{x}\,.
\end{matrix}
$$
As a consequence of these definitions, the operator $\vert \phi\rangle \langle \phi\vert$ is the orthogonal projector onto $\CC\,\phi$\,. 

\smallskip
Recall that a \emph{density matrix} $\r$ on some Hilbert space $\rH$ is a trace-class, positive operator such that 
$\tr(\r)=1\,.$ The convex set of all density matrices on $\rH$ will be denoted by $\rE(\rH)$. The extreme points of this convex set are the \emph{pure states}, that is, the rank one orthogonal projectors:
$$
\r=\vert\phi\rangle\langle\phi\vert\,,
$$
with $\phi\in\rH$, $\norme\phi=1$. 
The set of pure states on $\rH$ will be denoted by $\rS(\rH)$. 

\smallskip
Let $\rN$ stand for a finite or a countable set of indices.
If $\{A_i\,;\ i\in\rN\}$ is a family of bounded operators on $\rH$ such that
$$
\sum_{i\in\rN} A_i^*\, A_i=I\,,
$$
where the convergence above is understood for the weak topology,
then the mapping 
$$
\r\mapsto\rM(\r)=\sum_{i\in\rN} A_i\, \r\, A_i^*\,,
$$
is well-defined, for the series is $\norme{\cd}_1$-convergent, and the mapping
preserves the property of being a density matrix. It is a so-called \emph{completely positive map}. 

Note that such a completely positive map admits an adjoint map $\rM^*$ acting on the bounded operators on $\rH$. More precisely, the mapping
$$
\rM^*(X)=\sum_{i\in\rN} A_i^*\, X\,A_i\,,
$$
is a strongly convergent series and
satisfies
$$
\tr(\rM(\r)\, X)=\tr(\r\,\rM^*(X))
$$
for all density matrix $\r$ and all bounded operator $X$.

\section{Open Quantum Random Walks}\label{S:OQRW}

\subsection{General Setup}\label{SS:setup}
Let us explain here the setup in which we shall be working. It consists in special cases of Open Quantum Random Walks as described in \cite{APSS}, namely, the case of nearest neighbors, stationary quantum random walks on $\ZZ^d$. Our presentation here is slightly different of the one of \cite{APSS}, for we have adapted our notations to the simpler context that we are studying here.

\smallskip
On $\ZZ^d$ we consider the canonical basis $\{e_1,\ldots,e_d\}$ and we put 
$$
e_{d+j}=-e_j
$$ 
for all $j=1,\ldots, d$. For each $i\in\ZZ^d$ we denote by $N(i)$ the set of its $2d$ nearest neighbors, that is $N(i)=\{i+e_j\,;\ j=1,\ldots, 2d\}$.

We consider the space $\rK=\CC^{\ZZ^d}$, that is, any separable Hilbert space with an orthonormal basis indexed by $\ZZ^d$. We fix an orthonormal basis of $\rK$ which we shall denote by ${(\vert i\rangle)}_{i\in\ZZ^d}$. 
Let $\rH$ be a separable Hilbert space, it 
 stands for the space of degrees of freedom given at each point of $\ZZ^d$. \emph{In the rest of the article we always assume that $\rH$ is finite dimensional}. Consider the space $\rH\otimes \rK$. 

\smallskip
We are given a family $\{A_1,\ldots, A_{2d}\}$ of bounded operators on $\rH$ which satisfies
$$
\sum_{j=1}^{2d} A_j^*\,A_j=I\,.
$$
The idea is that the operator $A_j$ stands for the effect of passing from any point $i\in\ZZ^d$ to its neighbor $i+e_j$. The constraint above has to be understood as follows: ``the sum of all the effects leaving the site $i$ is $I$\,". It is the same idea as the one for transition matrices associated to Markov chains: ``the sum of the probabilities leaving a site $i$ is 1". 

To the family $\{A_1,\ldots, A_{2d}\}$ is then associated a completely positive map on $\rH$, namely:
$$
\rL(\rho)=\sum_{j=1}^{2d} A_j \,\rho\, A_j^*\,.
$$
To the family $\{A_1,\ldots,A_{2d}\}$ is also associated a completely positive map on $\rH\otimes\rK$ as follows. We put
$$
L_i^j=A_j\otimes \vert i+e_j\rangle\langle i\vert
$$
for all $i\in\ZZ^d$, all $j=1,\ldots,2d$. The operator $L_i^j$ emphasizes the idea that while one is passing from site $\vert i\rangle$ to its neighbor $\vert i+e_j\rangle$ in $\rK$, the effect on $\rH$ is the operator $A_j$. It is easy to check that
$$
\sum_{i\in\ZZ^d}\sum_{j=1}^{2d} {L_i^j}^*\,L_i^j=I\,,
$$
where the above series is strongly convergent. Hence, there is a natural completely positive map on $\rH\otimes \rK$ associated to these $L_i^j$'s, by putting
$$
\rM(\rho)=\sum_{i\in\ZZ^d}\sum_{j=1}^{2d} L_i^j\,\rho\,{L_i^j}^*
$$
for all density matrix $\rho$ on $\rH\otimes\rK$. Recall that the series above is convergent in trace-norm.  

In the following, we shall be interested in iterations $\rM^n$ of $\rM$ applied to density matrices of $\rH\otimes\rK$.  
We shall especially be interested in density matrices on $\rH\otimes\rK$ with the particular form
\begin{equation}\label{rho}
\r=\sum_{i\in\ZZ^d}\r_i\otimes \vert i\rangle\langle i\vert\,,
\end{equation}
where each $\r_i$ is not exactly a density matrix on $\rH$: it is a positive (and trace-class operator) but its trace is not 1. Indeed the condition that $\r$ is a state aims to
\begin{equation}\label{sumri}
\sum_{i\in\ZZ^d} \tr(\r_i)=1\,.
\end{equation} 
The reason for such a specialization is that any application of $\rM$ to any density matrix $\rho$ on $\rH\otimes\rK$ leads to a state of the form (\ref{rho}). This form (\ref{rho}) then stays 
stable under the dynamics. Hence the dynamics only deals with states of the form (\ref{rho}).

\smallskip
If $\r$ is a state on $\rH\otimes\rK$ of the form
$$
\r=\sum_i \r_i\otimes \vert i\rangle\langle i\vert\,,
$$
then a measurement of the ``position" in $\rK$, that is, a measurement along the orthonormal basis $(\vert i\rangle)_{i\in\rV}$, would give the value $\vert i\rangle$ with probability 
$$
p(i)=\tr(\r_i)\,.
$$
After applying the completely positive map $\rM$, the state of the system $\rH\otimes\rK$ can be easily checked to be
\begin{equation}\label{E:rLrho}
\rM(\r)=\sum_{i\in\ZZ^d} \left(\sum_{j=1}^{2d} A_j\, \r_{i-e_j}\, A_j^*\right)\otimes \vert iÊ\rangle\langle i\vert\,.
\end{equation}
Hence a measurement of the position in $\rK$ would give that each site $i$ is occupied with probability
\begin{equation}\label{E:pi}
p'(i)=\sum_{j=1}^{2d} \tr\left(A_j\,\r_{i-e_j}\,A_j^*\right)\,.
\end{equation}
And so on, by repeatedly applying $\rM$ to the initial state, we obtain a sequence of probability measures on $\ZZ^d$ which, in general, cannot be described in terms of a classical random walk. Indeed, the probability distribution at step $n+1$ cannot be deduced from the probability distribution at step $n$, we need to know the whole states $\r^{(n)}_i$ and not only their traces $\tr(\r^{(n)}_i)$.

\subsection{Examples}\label{SS:examples}

Let us illustrate the setup above, with some examples.

\smallskip
In the case $d=1$, we describe a quantum random walk on $\ZZ$ with the help of only two bounded operators $B$ and $C$ on $\rH$, satisfying
$$
B^*B+C^*C=I\,.
$$
The operator $B$ stands for the jumps to the left (it corresponds to the operator $A_2$ with the notations of previous subsection) and $C$ stands for the jumps to the right (it corresponds to the operator $A_1$).

Starting with an initial state $\r^{(0)}=\r_0\otimes \vert 0\rangle\langle 0\vert$, after one step we have the state
$$
\r^{(1)}=B\r_0B^*\otimes \vert \sm1\rangle\langle \sm1\vert+C\r_0C^*\otimes \vert 1\rangle\langle 1\vert\,.
$$
The probability of presence in $\vert\sm1\rangle $ is $\tr(B\r_0 B^*)$ and the probability of presence in $\vert 1\rangle$ is $\tr(C\r_0C^*)$. 

After the second step, the state of the system is
\begin{align*}
\r^{(2)}&=B^2\r_0{B^2}^*\otimes \vert\sm2\rangle\langle \sm2\vert+C^2\r_0{C^2}^*\otimes \vert 2\rangle\langle 2\vert+\\
&\ \ \ +\left(CB\r_0B^*C^*+BC\r_0C^*B^*\right)\otimes \vert 0\rangle\langle 0\vert\,.
\end{align*}
The associated probabilities for the presence in $\vert \sm2\rangle$, $\vert0\rangle$, $\vert 2\rangle$ are then 
$$
\tr(B^2\r_0{B^2}^*),\ \ \ \tr(CB\r_0B^*C^*+BC\r_0C^*B^*)\ \ \mbox{and}\ \ \tr(C^2\r_0{C^2}^*)\,,
$$ 
respectively. 

One can iterate the above procedure and generate our open quantum random walk on $\ZZ$.

\smallskip
As further example, take 
$$
B=\frac{1}{\sqrt3}\,\left(\begin{matrix}1&1\\0&1\end{matrix}\right)
\qq\mbox{and}\qq
C=\frac{1}{\sqrt3}\,\left(\begin{matrix}1&0\\-1&1\end{matrix}\right)\,.
$$
The operators $B$ and $C$ do satisfy $B^*B+C^*C=I$. Let us consider the associated open quantum random walk on $\ZZ$. Starting with the state 
$$
\r^{(0)}=\left(\begin{matrix}1&0\\0&0\end{matrix}\right)\otimes \vert0\rangle\langle0\vert\,,
$$
we find the following probabilities for the 4 first steps:
$$
\begin{matrix} 
&\vert-4\rangle&\vert-3\rangle&\vert-2\rangle&\vert-1\rangle&\vert0\rangle&\vert+1\rangle&\vert+2\rangle&\vert+3\rangle&\vert+4\rangle\\
n=0&&&&&1&&&&\\
n=1&&&&
\frac 13&&\frac 23&&&\\
n=2&&&
\frac 19&&\frac 39&& \frac 59&&\\
n=3
&&\frac 1{27}&& \frac 5{27}&&\frac{11}{27}&&\frac{10}{27}&\\
n=4&\frac 1{81}&&\frac{10}{81}&&\frac{27}{81}&&\frac{26}{81}&&\frac{17}{81}\\
\end{matrix} 
$$
The distribution obviously starts asymmetric, uncentered and rather wild. The interesting point is that, while keeping its quantum behavior time after time, simulations show up clearly a tendency to converge to a normal centered distribution.  Figure 1 below shows the distribution obtained at times $n=4$, $n=8$ and $n=20$.

\begin{figure}[h!]
\begin{center}
\leavevmode
{%
      \begin{minipage}{0.4\textwidth}
        \includegraphics[width=4cm,height=3.2cm]
        {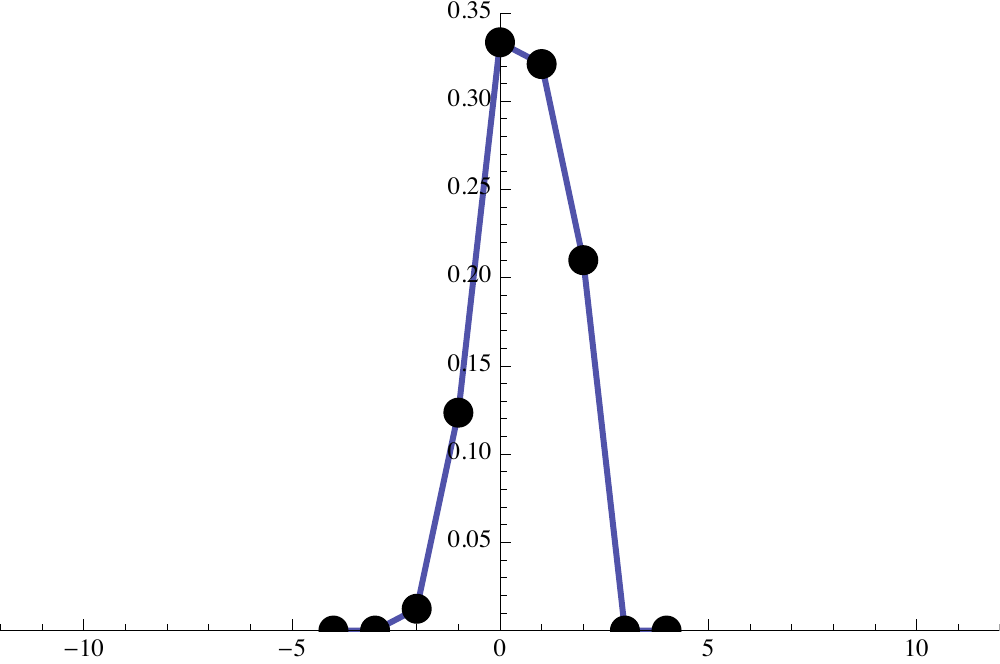}\\
     \vspace{-0.9cm} \strut
      
        \end{minipage}}\hspace*{-0.8cm}
{%
      \begin{minipage}{0.4\textwidth}
        \includegraphics[width=4cm,height=3.2cm]
         {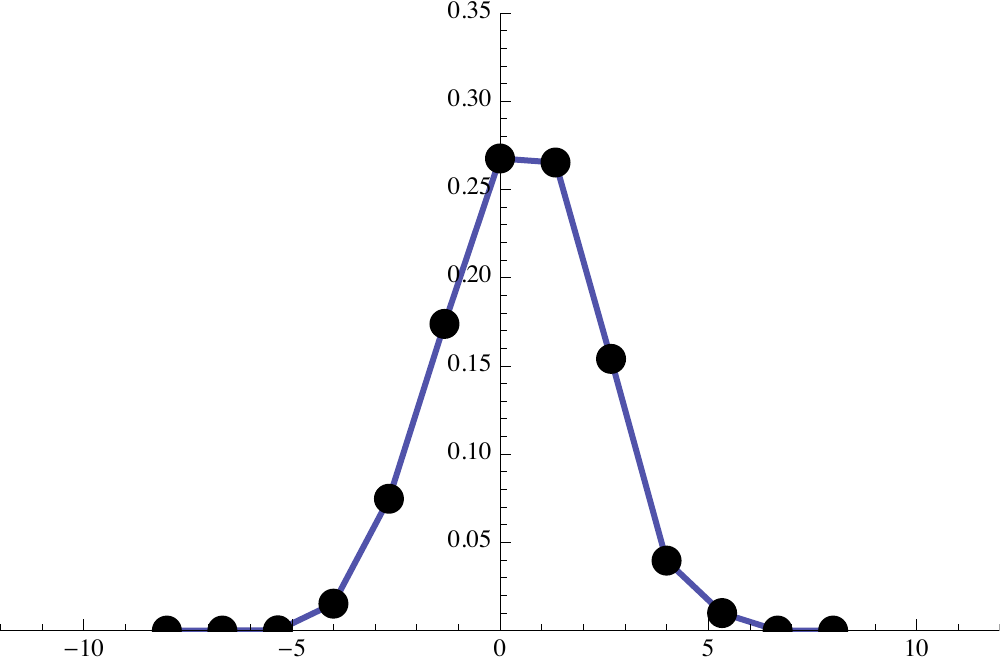}\\
        \vspace{-0.9cm} \strut
        \end{minipage}}\hspace*{-0.8cm}
{%
      \begin{minipage}{0.4\textwidth}
        \includegraphics[width=4cm,height=3.2cm]
        {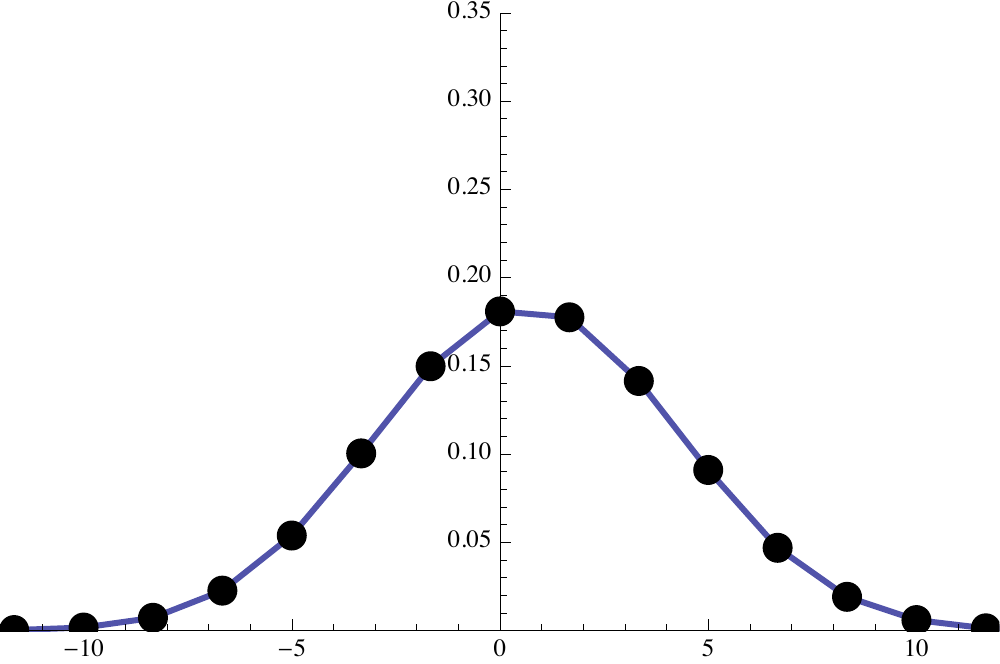}\\
      \vspace{-0.9cm} \strut
        \end{minipage}}
\end{center}
\vskip -0.1cm
\caption{\it An O.Q.R.W. on $\ZZ$ which gives rise to a centered Gaussian at the limit, while starting clearly uncentered (at times 
$n=4$, $n=8$, $n=20$)}
\end{figure} 

\bigskip
A much more trivial example on $\ZZ$ is obtained by taking
$$
B=\left(\begin{matrix}0&\sqrt p\\0&0\end{matrix}\right)\qq\mbox{and}\qq C=\left(\begin{matrix} 1&0\\ 0&\sqrt{1-p}\end{matrix}\right)\,,
$$
It is easy to compute the associated quantum trajectories and to show that they have the behavior of a random walk which goes straight to the right, with only one possible random jump to the left. This example will illustrate our Central Limit Theorem for the particular case where the Gaussian is degenerate. 

\bigskip
It is easy to produce Open Quantum Random Walks on $\ZZ^2$ by specifying 4 matrices $N,W,S,E$ on $\rH$ which satisfy 
\begin{equation}\label{E:2DCP}
N^*N+W^*W+S^*S+E^*E=I\,.
\end{equation}
Then, we ask the random walk to jump from any site to the four nearest neighbors, following $N$, $W$, $S$ or $E$, respectively.

\smallskip
One can for example combine two 1-dimensional Open Quantum Random Walks by asking them to act on the different coordinate axis. For example, take
$$
N=\sqrt\l\,\frac{1}{\sqrt3}\,\left(\begin{matrix}1&1\\0&1\end{matrix}\right)
\qq\mbox{and}\qq S=\sqrt\l\, \frac{1}{\sqrt3}\,\left(\begin{matrix}1&0\\-1&1\end{matrix}\right)
$$
together with
$$
W=\sqrt{(1-\l)}\,\left(\begin{matrix}0&\a\\0&\b\end{matrix}\right)
\ \mbox{and}\ 
E=\sqrt{(1-\l)}\,\left(\begin{matrix}1&0\\0&\g\end{matrix}\right)\,,
$$
with $\a^2+\b^2+\g^2=1$ and for some $\l\in [0,1]$. 

One can obtain behaviors with a single Gaussian, as in Figure 2, with $\l=3/4$, $\a=1/4$, $\b=1/4$.

\begin{figure}[h!]
\begin{center}
\vskip -0.3cm
\leavevmode
\includegraphics[width=6cm,height=6cm]{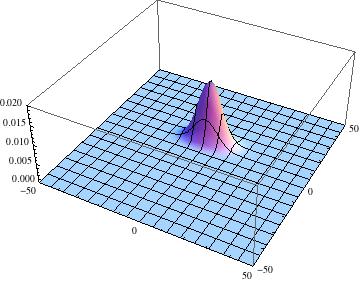}\\
\end{center}
\vskip -0.2cm
\caption{\it An O.Q.R.W. on $\ZZ^2$ which exhibits a single Gaussian asymptotically (at time $n=50$)}
\end{figure} 

\bigskip
The aim of the theorems to come now are to prove such Central Limit Theorems and to identify the elements of the limiting Gaussian distribution.

\section{Quantum Trajectories}\label{S:traj}

\subsection{Simulation of O.Q.R.W.}\label{SS:simulations}
Open Quantum Random Walks have the very nice property to admit a \emph{quantum trajectory approach}, that is, a classical process simulating the evolution of the density matrix. This approach to Open Quantum Random Walks is the one that allows us to prove a central limit theorem. Let us explain here this approach.

\smallskip
Starting from any initial state $\r$ on $\rH\otimes\rK$ we apply the mapping $\rM$ and then a measurement of the position in $\rK$, following the axioms of Quantum Mechanics. We end up with a random result for the measurement and a reduction of the wave-packet gives rise to a random state on $\rH\otimes\rK$ of the form
$$
\r_i\otimes \vert i\rangle\langle i\vert\,.
$$
We then apply the procedure again: an action of the mapping $\rM$ and a measurement of the position in $\rK$. The following result is proved in \cite{APSS}.

\begin{theorem}\label{T:quantum_traj}
By repeatedly applying the completely positive map $\rM$ and a measurement of the position on $\rK$, one obtains a sequence of random states on $\rH\otimes\rK$. This sequence is an homogenous Markov chain with law being described as follows. If the state of the chain at time $n$ is $\o^{(n)}=\r\otimes \vert i\rangle\langle i\vert$, then at time $n+1$ it jumps to one of the values
$$
\o^{(n+1)}=\frac{1}{p(j)}\,A_j\, \r \,{A_j}^*\otimes \vert i+e_j\rangle\langle i+e_j\vert\,,\ \ j=1,\ldots,2d,
$$
with probability
$$
p(j)=\tr\left(A_j \,\r\, {A_j}^*\right)\,.
$$
This Markov chain $(\o^{(n)})$
is a simulation of the master equation driven by $\rM$, that is,
$$
\EE\left[\o^{(n+1)}\,\vert\, \o^{(n)}\right]=\rM(\o^{(n)})\,.
$$
Furthermore, if the initial state is a pure state, then $\o^{(n)}$ stays valued in pure states and the Markov chain is described as follows. If the state of the chain at time $n$ is the pure state $\vert \varphi\rangle\otimes \vert i\rangle$, then at time $n+1$ it jumps to one of the values
$$
\frac{1}{\sqrt{p(j)}}\, A_j\, \vert\varphi\rangle \otimes \vert i+e_j\rangle\,,\ \ i\in\rV,
$$
with probability
$$
p(j)=\normca{A_j \,\vert\varphi\rangle}\,.
$$
\end{theorem}
In a more usual probabilistic language, this means that we have a Markov chain $(\r_n,X_n)_{n\in\NN}$ with values in $\rE(\rH)\times\ZZ^d$ which is described as follows: from any position $(\r,X)$ one can only jump to one of the 2d different values
$$
\left(\frac{1}{p(j)}\,A_j\, \r \,{A_j}^*\,,\,X+e_j\right)
$$
with probability 
$$
p(j)=\tr\left(A_j \,\r\, {A_j}^*\right)\,.
$$
What Theorem \ref{T:quantum_traj} says is that the law of the random variable $X_n$ coincides with the distribution on $\ZZ^d$ of our open quantum random walk at time $n$, when starting with the initial state
$\r_0\otimes\vert X_0\rangle\langle X_0\vert$.

Theorem \ref{T:quantum_traj} also says that if the initial condition is in $\rS(\rH)\otimes\ZZ^d$ then the Markov chain always stays in $\rS(\rH)\otimes\ZZ^d$.

\subsection{Ergodic Property}\label{SS:ergodic}

We now recall an ergodic theorem for quantum trajectories, as proved in \cite{K-M}, that we adapt to our context and notations. Recall the completely positive map on $\rH$ associated to the operators $A_1,\ldots,A_{2d}$\,:
$$
\rL(\r)=\sum_{i=1}^{2d}A_i\,\r\,A_i^*\,.
$$

\begin{theorem}\label{T:ergodic}
If $(\r_n,X_n)$ is the Markov chain obtained by the quantum trajectory procedure as in Theorem \ref{T:quantum_traj} then the sequence
$$
\frac 1n \sum_{i=1}^n \r_i
$$
converges almost surely to a random variable $\theta_\infty$ which is valued in the set of invariant states for $\rL$. 

In particular, if $\rL$ admits a unique invariant state $\r_\infty$, then the above Cesaro mean converges almost surely to $\r_\infty$.
\end{theorem}

\section{The Central Limit Theorem}\label{S:CLT}

\subsection{The main Theorem}\label{SS:main}

In this section we make the following hypothesis on $\rL$\,:

\bigskip\noindent
\ \ (H1)\ \ :\ \  \emph{$\rL$ admits a unique invariant state $\r_\infty$.}

\bigskip\noindent
We start with some notations. We put
$$
\boxed{\ \ m=\sum_{i=1}^{2d}\tr(A_i\,\r_\infty\, A_i^*)\, e_i\ \ }
$$
which is an element of $\RR^d$.

In the following we shall denote by $x\cdot y$ the usual scalar product on $\RR^d$. We denote by $m_i$, $i=1,\ldots d$, the coordinates of $m$ in $\RR^d$, that is $m_i=m\cdot e_i$ for $i=1,\ldots, d$. 

\begin{lemma}\label{existe_L} For every $l\in\RR^d$, the equation 
\begin{equation}\label{E:eq_L}
\boxed{\ \ \left(L-\rL^*(L)\right)=\sum_{i=1}^{2d} A_i^*A_i\, (e_i\cdot l)-(m\cdot l)\, I\ \ }
\end{equation}
admits a solution. The difference between any two solutions of \eqref{E:eq_L} is a multiple of the identity.
\end{lemma}
\begin{proof}
By definition of $m$ we have, for every $l\in\RR^d$
$$
\sum_{i=1}^{2d}\tr\left(A_i\,\r_\infty\, A_i^*\right)\, e_i\cdot l=m\cdot l\,,
$$
hence
$$
\tr\left(\r_\infty\left(\sum_{i=1}^{2d} A_i^*A_i\, (e_i\cdot l)-(m\cdot l)\, I\right)\right)=0\,.
$$
We have proved that $\sum_{i=1}^{2d} A_i^*A_i\, (e_i\cdot l)-(m\cdot l)\, I$ belongs to $\{\r_\infty\}^\perp$. But $\{\r_\infty\}^\perp$ is equal to $\ker(I-\rL)^\perp$, by Hypothesis (H1). Furthermore $\ker(I-\rL)^\perp$ is equal to the range of $I-\rL^*$. We have proved that 
  $\sum_{i=1}^{2d} A_i^*A_i\, (e_i\cdot l)-(m\cdot l)\, I$ belongs to the range of $I-\rL^*$. This gives the announced existence.
  
If $L'$ is any other solution of \eqref{E:eq_L} then, putting $H=L-L'$ gives
$$
H-\rL^*(H)=0\,.
$$
This is to say that $H$ is an eigenvector of $\rL^*$ for the eigenvalue 1. By the hypothesis (H1), the eigenspace of $\rL$ for the eigenvalue 1 is of dimension 1. Hence the eigenspace of $\rL^*$ for the same eigenvalue is also 1-dimensional. As we have $\rL^*(I)=I$,  this means that all eigenvectors of $\rL^*$ for the eigenvalue 1 are multiple of the identity. Hence $H$ is a multiple of the identity.
\end{proof}

\smallskip
In the following we shall denote by $L_l$ a solution of (\ref{E:eq_L}) associated to $l\in\RR^d$. In the case where $l=e_i$, for some $i=1,\ldots, d$,  we denote $L_l$ by $L_i$ simply. In terms of the coordinates $(l_i)$ of $l$, note that we have
$$
L_l=\sum_{i=1}^d l_i\,L_i\,.
$$
We can now formulate our main Central Limit Theorem.
 
\begin{theorem}\label{T:main}
Consider the stationary open quantum random walk on $\ZZ^d$ associated to the operators $\{A_1,\ldots, A_{2d}\}$. We assume that the completely positive map 
$$
\rL(\r)=\sum_{i=1}^{2d} A_i\,\r\,A_i^*
$$
admits a unique invariant state $\r_\infty$. 
Let $(\r_n,X_n)_{n\geq 0}$ be the quantum trajectory process associated to this open
 quantum random walk,  then 
 $$
 \lim_{n\to \infty} \frac{X_n}{n} = m, \;\;\; \hbox{ a.s. }
 $$
  and
$$
\frac{X_n-n\,m}{\sqrt n}
$$
converges in law to the Gaussian distribution $\rN(0,C)$ in $\RR^d$, with covariance matrix
$$
\fbox{\begin{minipage}{0.5\textwidth}
\begin{align*}
C_{ij}&=\delta_{ij}\,\left(\tr(A_i\,\r_\infty\, A_i^*)+\tr(A_{i+d}\,\r_\infty\,A_{i+d}^*)\right)-m_im_j+\\
&\ \ \ +\left(\tr(A_i\,\r_\infty\, A_i^*\, L_j)+\tr(A_j\,\r_\infty\, A_j^*\, L_i)\right.\\
&\ \ \ \left.-\tr(A_{i+d}\,\r_\infty\, A_{i+d}^*\, L_j)-\tr(A_{j+d}\,\r_\infty\, A_{j+d}^*\, L_i)\right)\\
&\ \ \ -\left(m_i\tr(\r_\infty\, L_j)+m_j\tr(\r_\infty\, L_i)\right)\,.
\end{align*}
\end{minipage}}
$$
\end{theorem}
\begin{proof}
Consider the Markov chain $(\r_n,X_n)_{n\in\NN}$, with values in $\rE(\rH)\times\ZZ^d$, associated to the quantum trajectories of $\rM$. We put $\NN^{*}=\NN\setminus\{0\}$ and $\D X_n=X_n-X_{n-1}$, for all $n\in\NN^*$ and we consider the stochastic process $(\r_n,\D X_n)_{n\in\NN^*}$ which is also a Markov chain, but with values in $\rE(\rH)\times\{e_1,\ldots,e_{2d}\}$. Its transition probabilities are given by
$$
P\left[(\r,e_i);(\r',e_j)\right]=\begin{cases}\tr\left(A_j\,\r\,A_j^*\right)
&
\mbox{if }
\r'=\frac{A_j\,\r\,A_j^*}{\tr(A_j\,\r\,A_j^*)}\,,\\
0&\mbox{otherwise},
\end{cases}
$$
for all $i,j\in\{1,\ldots,2d\}$. 

We are given a fixed $l\in\RR^d$ and we wish to write a Central Limit Theorem for $(X_n\cdot l)_{n\in\NN}$. Our first step is to find a solution to the so-called \emph{Poisson equation}, that is, we wish to find a function $f$ on $\rE(\rH)\times\{e_1,\ldots,e_{2d}\}$ such that 
\begin{equation}\label{E:poisson}
(I-P)f(\r,x)=x\cdot l-m\cdot l\,.
\end{equation}

\begin{lemma}\label{L:form_f}
A solution of (\ref{E:poisson}) is given by 
\begin{equation}\label{E:sol_Poisson}
f(\r,x)=\tr(\r\, L_l)+x\cdot l\,.
\end{equation}
\end{lemma}
\begin{proof}[of Lemma \ref{L:form_f}]

If we define $f$ by
$$
f(\r,x)=\tr(\r\, L_l)+x\cdot l\,,
$$
we get
\begin{multline*}
(I-P)f(\r,x)=\tr(\r\,L_l)+x\cdot l-\left(\sum_{i=1}^{2d} \tr(A_i\,\r\,A_i^*\, L_l)+\right.\hfill\\
\hfill\left.+\sum_{i=1}^{2d} \tr(A_i\, \r\, A_i^*)\, e_i\cdot l\right)\\
\hphantom{(I-P)f(\r,x)}=\tr\left(\r\, \left(\left(L_l-\rL^*(L_l)\right)-\sum_{i=1}^{2d} A_i^*A_i\, e_i\cdot l\right)\right)+x\cdot l\hfill\\
\hphantom{(I-P)f(\r,x)}=-m\cdot l+x\cdot l\,.\hfill
\end{multline*}
That is, the function $f$ is a solution of the Poisson equation.
\end{proof}[of Lemma]

\smallskip
The second step of the proof consists in carrying the problem of our central limit theorem to a central limit theorem for a martingale.

With the help of the Poisson equation, we have
\begin{align*}
X_n\cdot l-n(m\cdot l)&=X_0\cdot l+\sum_{k=1}^n\left((X_k-X_{k-1})-m\right)\cdot l\\
&=X_0\cdot l+\sum_{k=1}^n(I-P)f(\r_k,\D X_k)\\
&=X_0\cdot l+\sum_{k=2}^n\left(f(\r_k,\D X_k)-Pf(\r_{k-1},\D X_{k-1})\right)\\
&\quad+f(\r_1,\D X_1)-Pf(\r_n,\D X_n)\,.
\end{align*}
We put 
$$
M_n=\sum_{k=2}^nf(\r_k,\D X_k)-Pf(\r_{k-1},\D X_{k-1})\,.
$$ 
Clearly $(M_n)_{n\geq2}$ is a centered martingale, with respect to the filtration $(\rF_n)_{n\geq 2}$, where $\rF_n=\s\{(\r_k,X_k)\,;\ k\leq n\}$. 
Indeed, 
$$\EE\left[\D M_n\,\vert\, \rF_{n-1}\right]=\EE\left[f(\r_n,\D X_{n})\,\vert\, (\r_{n-1},\D X_{n-1})\right]- Pf(\r_{n-1},\D X_{n-1})=0$$
from the definition of $P$.\\*
We put
$$
R_n=X_0\cdot l+f(\r_1,\D X_1)-Pf(\r_n,\D X_n)\,.
$$
We claim that $(\ab{R_n})_{n\in\NN^*}$ is bounded. Indeed, by Equations (\ref{E:poisson}) and (\ref{E:sol_Poisson}) we have
$$
Pf(\r_n,\D X_n)=\tr(\r_n\, L_l)+m\cdot l
$$
and $\ab{\tr(\r_n\, L_l)}$ is bounded independently of $n$ by 
$$
\norme{\r_n}_1\,\norme{L_l}_\infty=\norme{L_l}_\infty\,.
$$
This means that the term $R_n$ has no contribution to  the law of large number or to the central limit theorem. 
It is thus sufficient to obtain a  law of large number and a central limit theorem for the martingale $(M_n)_{n\in\NN^*}$. We recall the form of the Central Limit Theorem for martingales that we shall use here.

\begin{theorem}[cf \cite{H-H}, Theorem 3.2 and Corollary 3.1]
Let $(M_n)_{n\in\NN}$ be a centered, square integrable, real martingale for the filtration $(\rF_n)_{n\in\NN}$. If, for all $\e>0$, we have the following convergences in probability:
\begin{equation}\label{E:cond1}
\lim_{n\rightarrow+\infty} \frac1n\sum_{k=1}^n \EE\left[(\D M_k)^2\, \indic_{\ab{\D M_k}\geq \e\sqrt n}\,\vert\, \rF_{k-1}\right]=0
\end{equation}
and
\begin{equation}\label{E:cond2}
\lim_{n\rightarrow+\infty} \frac1n\sum_{k=1}^n \EE\left[(\D M_k)^2\,\vert\, \rF_{k-1}\right]=\s^2
\end{equation}
for some $\s\geq0$, then $M_n/\sqrt n$ converges in distribution to a $\rN(0,\s^2)$ distribution.
\end{theorem}

\smallskip
As a third step of our proof we shall prove that $(M_n)_{n\geq2}$ satisfies the property (\ref{E:cond1}). We have 
\begin{align*}
\D M_k&=f(\r_k,\D X_k)-Pf(\r_{k-1},\D X_{k-1})\\
&=\tr(\r_k\,L_l)+\D X_k\cdot l-m\cdot l-\tr(\r_{k-1}\cdot L_l)\,.
\end{align*}
In particular $\D M_k$ is bounded independently of $k$ for
\begin{align}
\nonumber
\ab{\D M_k}&\leq \norme{\r_k}_1\,\norme{L_l}_\infty+\norme{\D X_k}\, \norme l+\norme m\, \norme l+\norme{\r_{k-1}}_1\, \norme{L_l}_\infty\\
\label{bounded-increments}
&\leq 2\norme{L_l}_\infty+\norme l+\norme m\, \norme l\,.
\end{align}
 Concerning the law of large number, since $M_n$ has bounded increments it implies that $M_n/n\to 0$ a.s. by Azuma's inequality
and Borel Cantelli lemma. This implies the law of large numbers for $(X_n)$ since $\vert R_n\vert$ is bounded.

 Remark now that the condition (\ref{E:cond1}) is then obviously satisfied as $\indic_{\ab{\D M_k}\geq \e\sqrt n}$ vanishes for $n$ large enough.

\smallskip
The fourth step of the proof consists in computing the quantity
$$
\EE\left[(\D M_k)^2\,\vert\, \rF_{k-1}\right]\,,
$$
in order to verify that Condition (\ref{E:cond2}) is satisfied. We have
$$
\D M_k=\tr(\r_k\,L_l)-\tr(\r_{k-1}\,L_l)+(\D X_k-m)\cdot l
$$
so that
\begin{align*}
(\D M_k)^2&=\tr(\r_k\,L_l)^2-\tr(\r_{k-1}\, L_l)^2\\
&\qquad-2\tr(\r_{k-1}\, L_l)\left[\tr(\r_k\,L_l)-\tr(\r_{k-1}\,L_l)+(\D X_k-m)\cdot l\right]\\
&\qquad+(\D X_k\cdot l-m\cdot l)^2+2\tr(\r_k\, L_l)\,(\D X_k\cdot l-m\cdot l)\,.
\end{align*}
We denote by $T_1$, $T_2$ and $T_3$, respectively, the three lines appearing in the right hand side above. The term $\EE[T_1\,\vert\, \rF_{k-1}]$ is equal to
$$
\EE[\tr(\r_k\,L_l)^2\,\vert\,\rF_{k-1}]-\tr(\r_k\,L_l)^2+\tr(\r_k\,L_l)^2-\tr(\r_{k-1}\, L_l)^2\,.
$$
The term $\EE[\tr(\r_k\,L_l)^2\,\vert\,\rF_{k-1}]-\tr(\r_k\,L_l)^2$ is the increment of a martingale $(Y_n)$ and it is bounded independently of $k$ (using the same kind of estimates as for $\ab{R_n}$ above). Hence $Y_n/n$ converges almost surely to 0. 

The term $\tr(\r_k\,L_l)^2-\tr(\r_{k-1}\, L_l)^2$, when summed up to $n$ gives $\tr(\r_n\,L_l)^2-\tr(\r_1\,L_l)^2$ and hence converges to 0 when divided by $n$. 

\smallskip
The term $\EE[T_2\,\vert\,\rF_{k-1}]$ clearly vanishes for it makes appearing the conditional expectation of the increment of the martingale $(M_n)$. 

\smallskip
We finally compute $\EE[T_3\,\vert\,\rF_{k-1}]$. We get
\begin{multline*}
\EE[T_3\,\vert\,\rF_{k-1}]=\EE\left[\left(\D X_k\cdot l\right)^2-2(m\cdot l)(\D X_k\cdot l)+(m\cdot l)^2+\right.\hfill\\
\hfill\left.+2\tr(\r_k\,L_l)\left(\D X_k\cdot l-m\cdot l\right)\,\vert\, \rF_{k-1}\right]\\
\hphantom{\EE[T_3\,\vert\,\rF_{k-1}]\ \ }=\sum_{i=1}^{2d}\tr\left(A_i\, \r_{k-1}\,A_i^*\right)\,\left[(e_i\cdot l)^2-2(m\cdot l)(e_i\cdot l)\right]+\hfill\\
\hfill+2\sum_{i=1}^{2d} \tr\left(A_i\,\r_{k-1}\, A_i^*\, L_l\right)\,(e_i\cdot l-m\cdot l)+(m\cdot l)^2\\
\hphantom{\EE[T_3\,\vert\,\rF_{k-1}]\ \ }=\tr\left(\r_{k-1}\,\left(\sum_{i=1}^{2d} A_i^*A_i\left(e_i\cdot l-m\cdot l\right)^2+\right.\right.\hfill\\
\hfill\left.\left.\vphantom{\sum_{i=1}^{2d}}+2 A_i^*\, L_l\, A_i\, (e_i\cdot l-m\cdot l)\right)\right)\,.
\end{multline*}
We put
$$
\Gamma_l=\sum_{i=1}^{2d} A_i^*A_i\left(e_i\cdot l-m\cdot l\right)^2+2 A_i^*\, L_l\, A_i\, (e_i\cdot l-m\cdot l)\,.
$$

\smallskip
Putting everything together, by the fact that $Y_n/n$ converges to 0 and by the Ergodic Theorem \ref{T:ergodic}, we get that 
$$
\frac 1n \sum_{k=3}^n\EE\left[\left(\D M_k\right)^2\,\vert\,\rF_{k-1}\right]
$$
converges almost surely to 
$$
\s_l^2=\tr(\r_\infty\, \Gamma_l)\,.
$$

\smallskip
The fifth and last step of the proof consists in rewriting the variance $\s_l^2$ in order to make the covariance matrix $C$ appearing. We have
\begin{align*}
\Gamma_l&=\sum_{i=1}^{2d} A_i^*A_i(e_i\cdot l-m\cdot l)^2+2\sum_{i=1}^{2d} A_i^*\,L_l\,A_i(e_i\cdot l-m\cdot l)\\
&=\sum_{i=1}^{2d} A_i^*A_i(e_i\cdot l)^2-2(m\cdot l)\sum_{i=1}^{2d} A_i^*A_i (e_i\cdot l)+(m\cdot l)^2+\\
&\qquad+2\sum_{i=1}^{2d} A_i^*\, L_l\,A_i(e_i\cdot l)-2(m\cdot l)\rL^*(L_l)\,.
\end{align*} 
Hence, this gives
\begin{align*}
\tr(\r_\infty\,\Gamma_l)&=\sum_{i=1}^{2d} \tr(A_i\,\r_\infty\, A_i^*)(e_i\cdot l)^2-2(m\cdot l)^2+(m\cdot l)^2+\\
&\qquad+2\sum_{i=1}^{2d} \tr(A_i\,\r_\infty\,A_i^*\, L_l)(e_i\cdot l)-2(m\cdot l)\tr(\rL(\r_\infty)\, L_l)\\
&=-(m\cdot l)^2+\sum_{i=1}^{2d} \tr(A_i\,\r_\infty\, A_i^*)(e_i\cdot l)^2+2\sum_{i=1}^{2d} \tr(A_i\,\r_\infty\, A_i^*\, L_l)(e_i\cdot l)\\
&\qquad-2(m\cdot l)\tr(\r_\infty\, L_l)\,.
\end{align*}
This gives
\begin{align*}
\s_l^2&=-\sum_{i,j=1}^{d} m_im_jl_il_j+\sum_{i=1}^{d} l_i^2\left(\tr(A_i\,\r_\infty\,A_i^*)+\tr(A_{i+d}\, \r_\infty\, A_{i+d}^*)\right)+\\
&\qquad+ 2\sum_{i,j=1}^d l_il_j\left(\tr(A_i\,\r_\infty\, A_i^*\, L_j)-\tr(A_{i+d}\,\r_\infty\,A_{i+d}^*\, L_j)\right)\\
&\qquad-2\sum_{i,j=1}^d l_il_jm_i\tr(\r_\infty L_j)\,.
\end{align*}
This proves that 
$$
\s_l^2=\sum_{i,j=1}^d l_il_j\, C_{ij}\,,
$$
where the matrix $\s$ is the one given in the theorem statement. The central limit theorem is proved.

\end{proof}

\smallskip
Note that here appears a key point in our proof: all the quadratic terms in $\r_k$ disappear in the limit; this is crucial for otherwise it would have been impossible to handle them without information on the invariant measure of the Markov chain $(\r_n)$.

\subsection{The one dimensional case}
The one dimensional case is a useful one, we make simpler in this case the formulas we have obtained above. 

In the case where the dimension is $d=1$, there are only two jump operators $A_1$ and $A_2$, which satisfy
$$
A_1^*A_1+A_2^*A_2=I\,.
$$
We have
$$
m=\tr(A_1\,\r_\infty\, A_1^*)-\tr(A_2\,\r_\infty\, A_2^*)\,.
$$
In dimension 1 there is only one operator $L_i$, the operator $L_1$, which we denote here by $L$ simply and which is solution of
$$
L-\rL^*(L)=A_1^*A_1-A_2^*A_2-mI=2A_1^*A_1-(1+m)I\,.
$$
Finally, following the theorem above, we have
$$m=1-2\tr(A_2\,\r_\infty\, A_2^*)\,$$
and
\begin{align*}
\s^2&=\Tr\left(A_1\r_\infty A_1^*+A_2\r_\infty A_2^*\right)-m^2+\\
&\ \ \ +2\Tr\left[\left(A_1\r_\infty A_1^*-A_2\r_\infty A_2^*\right)L\right]-2m\Tr(\r_\infty L)\\
&=1-m^2-2m\Tr(\r_\infty L)+2\Tr\left[\left(A_1\r_\infty A_1^*-A_2\r_\infty A_2^*\right)L\right]\\
&=1-m^2-2m\Tr(\r_\infty L)+2\Tr(\r_\infty L)-4\Tr\left[\left(A_2\r_\infty A_2^*\right)L\right]\\
&=1-m^2+2(1-m)\Tr(\r_\infty L)-4\Tr\left[\r_\infty A_2^*LA_2\right]\\
&=1-m^2+4\left(\Tr\left(A_2\r_\infty A_2^*\right)\Tr(\r_\infty L)-\Tr\left(\r_\infty A_2^*LA_2\right)\right)\,,
\end{align*}
or equivalently
$$
\s^2=1-m^2+4\left(\Tr\left(\r_\infty A_1^*LA_1\right)-\tr\left(A_1\r_\infty A_1^*\right)\Tr(\r_\infty L)\right)\,.
$$

\subsection{Examples}

We shall now explore several examples in order to illustrate our Central Limit Theorem. Let us first start with two examples on $\ZZ$. The example 
$$
B=\frac{1}{\sqrt3}\,\left(\begin{matrix}1&1\\0&1\end{matrix}\right)
\qq\mbox{and}\qq
C=\frac{1}{\sqrt3}\,\left(\begin{matrix}1&0\\-1&1\end{matrix}\right)
$$
that we mentioned earlier falls in the scope of our theorem for it admits a unique invariant state
$$
\r_\infty=\frac12 I\,.
$$
In particular we have 
$$
m=\tr\left(C\,\r_\infty\, C^*\right)-\tr\left(B\,\r_\infty\,B^*\right)=0\,.
$$
We recover here that the limit Gaussian distribution is centered, as was observed in the simulations above.

The operator $L$, given by Lemma \ref{existe_L} is
$$
L=
\frac 13\,\left(
\begin{matrix}
 -5 & 1    \\
 1 & 0   
\end{matrix}
\right)+\l I\,.
$$
This gives 
$$
\s^2=\frac{8}9\,.
$$

\bigskip
Let us compute the case of our trivial example on $\ZZ$  obtained by taking
$$
B=\left(\begin{matrix}0&\sqrt p\\0&0\end{matrix}\right)\qq\mbox{and}\qq C=\left(\begin{matrix} 1&0\\ 0&\sqrt{1-p}\end{matrix}\right)\,.
$$
In that case the unique invariant state is 
$$
\r_\infty=\left(
\begin{matrix}
 1 & 0    \\
 0 & 0   
\end{matrix}
\right)\,.
$$
We find $m=1$ in that case, which is compatible with the behavior we described for this example. 

The operator $L$ in this case is
$$
L=\left(\begin{matrix}
 -2 & 0    \\
 0 & 0   
\end{matrix}
\right)+\l I\,.
$$
This gives $\s^2=0$. We recover that the asymptotic behavior of this open quantum random walk is degenerate, with drift +1.

\bigskip
Let us end up this illustration with the 2-dimensional example mentioned in Subsection \ref{SS:examples}:
$$
N=\frac12\,\left(\begin{matrix}1&1\\0&1\end{matrix}\right)\,,
\ \  S=\frac12\,\left(\begin{matrix}1&0\\-1&1\end{matrix}\right)\,,\ \ 
W=\frac18\,\left(\begin{matrix}0&1\\0&1\end{matrix}\right)\,,
\ \ 
E=\frac14\,\left(\begin{matrix}1&0\\0&\sqrt{\frac72}\end{matrix}\right)\,.
$$
We find a unique invariant state
$$
\r_\infty=\frac1{33}\,\left(\begin{matrix}17&0\\0&16\end{matrix}\right)\,.
$$
The average is 
$$
m=\left(\frac{29}{132},\frac{-1}{132}\right)\,.
$$
The two solutions of Equation (\ref{E:eq_L}) are then
$$
L_1=\left(\begin{matrix}0&\frac{68 \left(16+\sqrt{14}\right)}{3993}\\ \frac{68 \left(16+\sqrt{14}\right)}{3993}&\frac{8 \left(756+17 \sqrt{14}\right)}{3993}\end{matrix}\right)\,,\ \ \ 
L_2=\left(\begin{matrix}0&\frac{8 \left(16+\sqrt{14}\right)}{3993}\\\frac{8 \left(16+\sqrt{14}\right)}{3993}&\frac{4 \left(-57+4 \sqrt{14}\right)}{3993}\end{matrix}\right)\,.
$$
and we find the following covariance matrix
$$
C=\left(\begin{matrix}0.675&0.008\\0.008&0.211\end{matrix}\right)\,,
$$
approximately\,.

\section{Application to Quantum Measurement \allowbreak Re\-cords}

In this section we leave for a while the setup of Open Quantum random Walks in order to show that our Central Limit Theorem actually applies to a wider situation: the recording of successive measurements in quantum trajectories. 

\subsection{Quantum Trajectory Setup}

The setup we shall present here is the one of recording quantum trajectories in discrete time, we actually speak of repeated measurements. This setup of \emph{Repeated Quantum Measurements}, based on the \emph{Repeated Quantum Interaction} scheme developed in \cite{A-P},  has been introduced and studied mathematically in \cite{Pel1} and \cite{Pel2}; it corresponds to actual important physical experiments such as the ones performed by S. Haroche's team on the indirect observation of photons in a cavity (\cite{Har1}, \cite{Har2}). 

Very quickly resumed, the setup is the following. A quantum system $\rH_S$ is performing an interaction with a quantum environment which has the form of a
chain of identical copies of a quantum system $\rK$, that is,
$$
\rH_E=\bigotimes_{n\in\NN^*} \rK\,.
$$
The dynamics in between $\rH_S$ and $\rH_E$ is obtained as follows. The
small system $\rH_S$ interacts with the first copy $\rK$ of the chain
during an interval $[0,h]$ of time and following some Hamiltonian $H_{\rm{tot}}$ on
$\rH_S\otimes\rK$. That is, the two systems evolve together following
the unitary operator 
$$
U=e^{-ihH_{\rm{tot}}}\,.
$$
After this first interaction, the small system $\rH_S$ stops
interacting with the first copy and starts an interaction with the
second copy which was left unchanged until then. This second
interaction follows the same unitary operator $U$. And so on, the
small system $\rH_0$ interacts repeatedly with the elements of the
chain one after the other, following the same unitary evolution $U$. 

\smallskip
We are given an orthonormal basis $\{e_1,\ldots,e_n\}$ of $\rK$. 
Assume that the initial state in $\rK$ before the interaction is $\vert e_1\rangle\langle e_1\vert$, and the initial state of $\rH_S$ is $\r$. The  whole state after interaction is
$$
U\left(\r\otimes\vert e_1\rangle\langle e_1\vert\right)U^*\,.
$$
The quantum channel on $\rH$ associated to that evolution is then given by
$$
\rL(\r)=\tri_\rK\left(U\left(\r\otimes\vert e_1\rangle\langle e_1\vert\right)U^*\right)=\sum_{i=1}^n M_i\,\r\,M_i^*\,,
$$
where the $M_i$'s are given by $M_i=U^1_i$, the coefficients of the first column of $U$ seen as a block matrix in the basis $\{e_1,\ldots,e_n\}$ of $\rK$.

In particular notice that this setup is as general as possible, for any given quantum channel $
\rL(\r)=\sum_{i=1}^n M_i\,\r\,M_i^*$ on $\rH$ could be obtained this way, by choosing a unitary $U$ with prescribed first column.

\smallskip
Now performing a measurement of any observable $X$ of $\rK$ which is diagonal with respect to the basis $\{e_1,\ldots,e_n\}$ gives rise to $n$ different possible values, obtained with respective probability
$$
p_i=\tr\left(M_i\,\r\,M_i^*\right)\,.
$$
The state of the whole system after the corresponding measurement is then
$$
\frac{1}{p_i}\, M_i\,\r\,M_i^*\otimes \vert e_i\rangle\langle e_i\vert\,.
$$
Regarding only the system $\rH_S$, the resulting state is $\r_1=\frac{1}{p_i}\, M_i\,\r\,M_i^*$.

\smallskip
Now repeating the procedure, via the repeated interaction scheme, we see that we obtain a Markov chain $(\r_n, X_n)_{n\in\NN}$, where $\r_n$ evolves in the set of density matrices of $\rH_S$ and $X_n$ evolves in the set $\{e_1,\ldots,e_n\}$. The law of the Markov chain is described as follows: if the chain at time $n$ is at $(\r,X)$, then at time $n+1$ it jumps to one of the values
$$
\left(\frac{1}{p_i}\, M_i\,\r\,M_i^*, e_i\right)\,,
$$
$i=1,....,n$,
with respective probability 
$$
p_i=\tr\left(M_i\,\r\,M_i^*\right)\,.
$$
This is the so-called \emph{quantum trajectory} associated to the quantum channel $\rL$, as obtained by \emph{repeated interaction and repeated measurement scheme}. 

We are interested in the recording of the different random choices for this successive measurements. That is, we are looking at the sequence of values $(X_n)_{n\in\NN}$ and we wish to write a Central Limit Theorem for the associated random walk $S_n=\sum_{i=1}^n X_i$.

\subsection{T.C.L. for Quantum Measurement Recordings}

Comparing this setup to the one we have developed for the quantum trajectories associated to Open Quantum Random Walks shows that it is exactly the same as in Theorem \ref{T:quantum_traj}, for $d=n$ and
$$
A_1=M_1,\ldots, A_n=M_n, \qq\mbox{and\ \ \ }A_{n+1}=\ldots=A_{2n}=0\,.
$$
In order to apply our previous result, we make the same important assumption here:

\bigskip\noindent
\ \ (H1)\ \ :\ \  \emph{$\rL$ admits a unique invariant state $\r_\infty$.}

\bigskip\noindent

We put
$$
\boxed{\ \ m=\sum_{i=1}^{n}\tr(M_i\,\r_\infty\, M_i^*)\, e_i\in\RR^n\,.\ \ }
$$
The Lemma \ref{existe_L} applies.

\begin{lemma}\label{existe_L2} 
For every $l\in\RR^n$, the equation 
\begin{equation}\label{E:eq_L}
\boxed{\ \ \left(L-\rL^*(L)\right)=\sum_{i=1}^{n} M_i^*M_i\, (e_i\cdot l)-(m\cdot l)\, I\ \ }
\end{equation}
admits a solution. The difference between any two solutions of \eqref{E:eq_L} is a multiple of the identity.
\end{lemma}
\smallskip
As in our main theorem, we denote by $L_l$ a solution of (\ref{E:eq_L}) associated to $l\in\RR^d$. In the case where $l=e_i$, for $i=1,\ldots, d$,  we denote $L_l$ by $L_i$ simply. 

\smallskip
The Central Limit Theorem for measurement records now reads as follows, as a direct application of 
Theorem \ref{T:main2}.

\begin{theorem}\label{T:main2}
Consider the quantum channel 
$$
\rL(\r)=\sum_{i=1}^{n} M_i\,\r\,M_i^*\,,
$$
on $\rH$, which we assume to
admit a unique invariant state $\r_\infty$. Consider the quantum random walk $(S_n)_{n\in\NN}$ on $\ZZ^n$ associated to the successive measurements associated to the quantum trajectory of $\rL$. Let $m$ and the $L_i$'s be given as described above. Then  $\lim_{n\to \infty} S_n/n=m$ a.s. and
$$
\frac{S_n-n\,m}{\sqrt n}
$$
converges in law to the Gaussian distribution $\rN(0,C)$ in $\RR^n$, with covariance matrix
$$
\fbox{\begin{minipage}{0.5\textwidth}
\begin{align*}
C_{ij}&=\delta_{ij}\,\tr(M_i\,\r_\infty\, M_i^*)-m_im_j+\\
&\ \ \ +\left(\tr(M_i\,\r_\infty\, M_i^*\, L_j)+\tr(M_j\,\r_\infty\, M_j^*\, L_i)\right)\\
&\ \ \ -\left(m_i\tr(\r_\infty\, L_j)+m_j\tr(\r_\infty\, L_i)\right)\,.
\end{align*}
\end{minipage}}
$$
\end{theorem}

\subsection{Examples}

One of the simplest interesting example of a quantum trajectory simulating some quantum channel is the one associated to \emph{spontaneous emission}. In that model, the systems $\rH_S$ and $\rK$ are both two-level systems, that is, $\CC^2$. The Hamiltonian, in the simplest configuration, is 
$$
H_{\rm tot}=i\left(\begin{matrix} 0&1\\0&0\end{matrix}\right)\otimes\left(\begin{matrix} 0&0\\1&0\end{matrix}\right)-i\left(\begin{matrix} 0&0\\1&0\end{matrix}\right)
\otimes\left(\begin{matrix} 0&1\\0&0\end{matrix}\right)\,.
$$
The associated unitary evolution is 
$$
U=e^{-i h H_{\rm tot}}=\left(\begin{matrix}1&0&0&0\\0&\cos(h)&-\sin(h)&0\\0&\sin(h)&\cos(h)&0\\0&0&0&1\end{matrix}\right)\,.
$$
The quantum channel is 
$$
\rL(\r)=M_1\,\r\,M_1^*+M_2\,\r\,M_2^*
$$
with
$$
M_1=\left(\begin{matrix}1&0\\0&\cos(h)\end{matrix}\right)\,,\qq M_2=\left(\begin{matrix}0&\sin(h)\\0&0\end{matrix}\right)\,.
$$
In what follows we assume that $\cos^2(h)\not=1$, for otherwise the dynamics is completely trivial. 

This quantum channel has a unique invariant state 
$$
\r_\infty=\left(\begin{matrix}1&0\\0&0\end{matrix}\right)\,.
$$
The quantum trajectories are easy to describe. Given a state 
$$
\r=\left(\begin{matrix} \a&z\\\ol{z}&\b\end{matrix}\right)
$$ 
and a position $X$, the next measurement leads to the state
$$
\frac{1}{\a+\cos^2(h)\b} \left(\begin{matrix} \a&\cos(h)z\\\cos(h)\ol{z}&\cos^2(h)\,\b\end{matrix}\right)
$$
and the position $X+e_1$, with probability $\a+\cos^2(h)\b$, or to the state
$$
\left(\begin{matrix} 1&0\\0&0\end{matrix}\right)
$$
and the position $X+e_2$, with probability $\sin^2(h)\b$. 

In the case it reaches the second value above, the quantum trajectory will not change anymore, that is, the state remains 
$$
\left(\begin{matrix} 1&0\\0&0\end{matrix}\right)
$$
with probability 1, giving always a step along $e_1$ for the random walk.

Computing $m$ as described above gives 
$$
m=e_1\,,\qq L_1= \left(\begin{matrix} 1&0\\0&0\end{matrix}\right)\,,\qq L_2=\left(\begin{matrix} -1&0\\0&0\end{matrix}\right)\,.
$$
In the Central Limit Theorem, the covariance matrix is then the null one. Which is what could be expected, regarding the description we gave for the quantum trajectories of this quantum channel.

\bigskip
One can also compute the Central Limit Theorem with a less trivial example. Consider the quantum channel that we have already met
$$
\rL(\r)=B\,\r\,B^*+C\,\r\,C^*\,,
$$
with
$$
B=\frac{1}{\sqrt3}\,\left(\begin{matrix}1&1\\0&1\end{matrix}\right)
\qq\mbox{and}\qq
C=\frac{1}{\sqrt3}\,\left(\begin{matrix}1&0\\-1&1\end{matrix}\right)\,.
$$
We find
$$
m=\frac{1}{2}(e_1+e_2),\qq L_1=\frac{1}{6}\,\left(\begin{matrix}-5&1\\1&0\end{matrix}\right)\,,\qq L_2=
\frac{1}{6}\,\left(\begin{matrix}5&-1\\-1&0\end{matrix}\right)\,.
$$
This gives the covariance matrix
$$
C=\frac{2}{9}\,\left(\begin{matrix}1&-1\\-1&1\end{matrix}\right)\,.
$$

\section{The Block-Diagonal Case}\label{SS:block}

\subsection{The Main Theorem}
The Central Limit Theorem proved above does not concern the case where $\rL$ admits several invariant states. This is typically the case when the asymptotic behavior shows up several Gaussian contributions. The proof we have obtained above does not adapt to the general case. However, there is one situation, with several Gaussians which we are able to treat. Let us describe it now.

\smallskip
Consider the operators $A_1,\ldots,A_{2d}$ satisfying 
$$
\sum_{i=1}^{2d} A_i^*A_i=I\,,
$$
as previously. We now assume that there exists a decomposition
$$
\rH=E_1\oplus E_2\oplus\ldots\oplus E_N
$$
of $\rH$ into orthogonal subspaces such that all the $A_i$'s are block-diagonal with respect to this decomposition. That is, 
$$
A_i(E_j)\subset E_j
$$
for all $i=1,\ldots, 2d$, all $j=1,\ldots, N$. This hypothesis is denoted by (H1') in the rest of this section. We denote by $P_j$ the orthogonal projector onto $E_j$. Note that the condition above is equivalent to 
$$
P_jA_i=A_iP_j
$$
for all $i=1,\ldots,2d$, all $j=1,\ldots, N$.  We put
$$
A_i^{(j)}=A_iP_j=P_jA_iP_j\,.
$$
In the same way we denote by $\rL^{(j)}$ the completely positive map associated to the operators $(A_i^{(j)})_{i=1}^{2d}$. On each subspace $E_j$ we have
$$
\sum_{i=1}^{2d} {A^{(j)}_i}^*A^{(j)}_i=\sum_{i=1}^{2d} P_jA_i^*A_iP_j=P_j=I_{E_j}\,.
$$
If $\r$ is a density matrix on $\rH$ we put
$$
\r^{(j)}=P_j\r P_j\,.
$$

\smallskip
Let $\r$ be a density matrix and $\PP_\r$ the law of the Markov chain $(\r_n, X_n)_{n\geq 0}$ obtained as previously, by the quantum trajectories associated to the matrices $A_i$, starting with the initial state $\r$. Recall that
$$
\left(\r_{n+1}=\frac{A_i\r_nA_i^*}{\Tr(A_i\r_nA_i^*)}\,,\  X_{n+1}=X_n+e_i\right)
$$
with probability $\Tr(A_i\r_nA_i^*)$.

We put $p^{(j)}_n=\Tr(P_j\r_n)$. 

\begin{lemma}\label{lemma1}
The process $(p^{(j)}_n)_{n\geq 0}$ is a martingale for the filtration 
$$
\rF_n=\s\left((\r_k,X_k)\,,\ k\leq n\right)\,.
$$
\end{lemma}
\begin{proof}
We have
\begin{align*}
\EE_\r\left[p_{n+1}^{(j)}\,\vert\, \rF_n\right]&=\sum_i\Tr\left(P_jA_i\r_nA_i^*\right)\\
&=\sum_i \Tr\left(P_jA_i\r_nA_i^*P_j\right)\\
&=\sum_i \Tr\left(A_iP_j\r_nP_jA_i^*\right)\\
&=\sum_i\Tr\left(A_i^*A_i\r_n^{(j)}\right)\\
&=\Tr(\r_n^{(j)})=p_n^{(j)}\,.
\end{align*}
\end{proof}

 Since $(p_n^{(j)})$ is non-negative and bounded, it converges a.s. and in $L^1$ to a limit that we denote
$$
p_\infty^{(j)}=\lim p_n^{(j)}.
$$
Remark that $\sum_{j=1}^N p_\infty^{(j)}=1$ since $\sum_{j=1}^N p_n^{(j)}=1$ for all $n$.
As $(p_n^{(j)})$ is a martingale we can consider the associated Girsanov transform (that is, the $h$-process). We define $\PP_\r^{(j)}$ to be the law on the trajectories which is given, on the length $n$ trajectories by
$$
\PP_n^{(j)}=\frac{p_n^{(j)}}{p_0^{(j)}}\, \PP_n
$$
where $\PP_n$ is the law on the trajectories with length $n$. In other words
$$
\PP_\r^{(j)}=\frac{p_\infty^{(j)}}{p_0^{(j)}}\, \PP_\r\,.
$$

\begin{proposition} \label{Restriction}
Under the law $\PP_\r^{(j)}$ the sequence 
$$
\left(\frac{\r_n^{(j)}}{\Tr\left(\r_n^{(j)}\right)}\,,\, X_n\right)_{n\geq 0}
$$
has the law of the quantum trajectories associated to the family of operators $(A_i^{(j)})_{i=1,\ldots, 2d}$ and starting from the state $\r_0^{(j)}$. 
\end{proposition}
\begin{proof}
The sequence $p_n^{(j)}=\Tr(P_j\r_n)$, $n\in\NN$,  is a function of $(\r_n)$. The chain $(\r_n, X_n)$ under $\PP^{(j)}$ is thus a $h$-process of the initial chain for the harmonic function $p^{(j)}(\r)=\Tr(P_j\r)$. We thus have that $(\r_n, X_n)$ is a Markov chain under $\PP^{(j)}$ with transition probabilities:
$$
\begin{cases}
\r_{n+1}=\frac{A_i\r_nA_i^*}{\tr\left(A_i\r_nA_i^*\right)} & \\
X_{n+1}=X_n+e_i& 
\end{cases}
$$
with probability 
$$
\frac{p_{n+1}^{(j)}}{p_n^{(j)}}\, \Tr(A_i\r_nA_i^*)\,.
$$
But we have
\begin{align*}
\frac{p_{n+1}^{(j)}}{p_n^{(j)}}\, \Tr(A_i\r_nA_i^*)&=\frac{\Tr(P_j\r_{n+1})}{\Tr(P_j\r_n)}\, \Tr(A_i\r_nA_i^*)\\
&=\frac{\Tr(P_j A_i\r_nA_i^*)}{\Tr(P_j\r_n)}\\
&=\frac{\Tr(A_i^{(j)}\r_n^{(j)}{A_i^{(j)}}^*)}{\Tr(P_j\r_n)}\,.
\end{align*}
We see that the transition probabilities only depend on the component $\r_n^{(j)}$. If we consider the sequence
$$
\wt\r_n^{(j)}=\frac{\r_n^{(j)}}{\tr(\r_n^{(j)})}
$$
we have
$$
\begin{cases}
 \wt\r_{n+1}^{(j)}=\frac{A_i^{(j)}\wt\r_n^{(j)}{A_i^{(j)}}^*}{\Tr\left(A_i^{(j)}\wt\r_n^{(j)}{A_i^{(j)}}^*\right)}& \\
X_{n+1}=X_n+e_i & 
\end{cases}
$$
with probability
$$
\frac{\Tr(A_i^{(j)}\wt\r_n^{(j)}{A_i^{(j)}}^*)}{\Tr(\wt\r^{(j)}_n)}\,.
$$
This exactly means that the sequence $(\wt\r_n^{(j)},X_n)_{n\geq 0}$ under $\PP^{(j)}$ has the law of the quantum trajectories associated to the family $(A_i^{(j)})_{i=1}^{2d}$. 
\end{proof}

\medskip
We now make the following hypothesis. 

\smallskip\noindent
(H2) Each of the mappings $\rL^{(j)}$ admits a unique invariant state $\r_\infty^{(j)}$. 

\smallskip\noindent
We then put $m^{(j)}=(m_1^{(j)},\ldots, m_{2d}^{(j)})$ where $m_k^{(j)}=\Tr(A_k\r_\infty^{(j)} A_k^*)$.

\smallskip\noindent
(H3) The $m^{(j)}$'s are all different.

\smallskip
Under these hypotheses we have the following result.

\begin{theorem}Under the hypotheses (H1'), (H2) and (H3) we have the following properties.

\smallskip\noindent
1) For all $j=1,\ldots,N$,
$$
\PP_\r\left[p_\infty^{(j)}=1\right]=p_0^{(j)}=1-\PP_\r\left[p_\infty^{(j)}=0\right]\,,
$$
that is, the vector $\vect{p}_n=(p_1^{(1)},\ldots,p_n^{(N)})$ converges to $(0,\ldots, 0, 1_j,0,\ldots,0)$ with probability $p_0^{(j)}$ (note that $\sum_j p_0^{(j)}=1$).

\smallskip\noindent
2) Conditionally to $p_\infty^{(j)}=1$ (that is, under the measure 
$$
\PP_\r\left[\,\cdot\,\vert\, p_\infty^{(j)}=1\right]=\PP_\r^{(j)})
$$
we have that $(\wt\r_n^{(j)}, X_n)$ has the law of the quantum trajectories associated to the family of matrices $(A_i^{(j)})_{i=1}^{2d}$. In particular, under this conditional law, the process
$$
\frac{\left(X_n-nm^{(j)}\right)}{\sqrt n}
$$
converges in distribution to the Gaussian distribution $\rN\left(0, {C^{(j)}}\right)$,
where ${C^{(j)}}$ is given by the same formula as in Theorem \ref{T:main} but for the family $(A_i^{(j)})$. 
\end{theorem}

\smallskip
Note that the theorem above concretely means that the quantum trajectories in that case are a mixture of Open Quantum Random Walks of the form of Theorem \ref{T:main}. The associated stochastic process can be obtained as follows: with probability $p^{(0)}_j$ the process $(X_n)$ follows the law of the Open Quantum Random Walks with associated matrices $A_i^{(j)}$ and then satisfies the corresponding Central Limit Theorem with mean $m^{(j)}$ and covariance matrix $C^{(j)}$. 

\smallskip
\begin{proof}
By proposition \ref{Restriction}, we know that under $\PP_\r^{(j)}$ the sequence $(\wt\r_n^{(j)},X_n)$ has the law of the quantum trajectories associated to the family $(A_i^{(j)})$. As the mapping $\rL^{(i)}$ admits a unique invariant state we also know that if we consider
$N_n(i)$ to be the number of jumps $e_i$ made by the quantum trajectory up to time $n$, then  we have
$$
\lim \frac{1}n N_n(i)=m_i^{(j)}
$$
almost surely for the measure $\PP_\r^{(j)}$  using the law of large numbers for quantum measurements of Theorem \ref{T:main2}.
This implies that  the measures $\PP_\r^{(j)}$ are all singular since the $m^{(j)}$'s are all different by hypothesis (H3). Indeed, let 
$$
\aaa_j=\{ \lim \frac{1}n N_n=m^{(j)}\}.
$$
Then, if $j\neq j'$, obviously $\aaa_j\cap \aaa_{j'}=\emptyset$ and $\PP_\r^{(j)} (\aaa_j)=1$, $\PP_\r^{(j')} (\aaa_{j'})=1$.
Consider now the sets
$$
\O^{(j)}=\{p_\infty^{(j)}>0\}\,.
$$
Then, if $j\neq j'$, 
$$
\PP_\rho(\O_j\cap \O_{j'})=0.
$$ 
Indeed, otherwise since $
\PP_\r^{(j)}=\frac{p_\infty^{(j)}}{p_0^{(j)}}\, \PP_\r$, it would imply that $\PP_\r^{(j)}(\O_j\cap \O_{j'})>0$
and $\PP_\r^{(j')}(\O_j\cap \O_{j'})>0$. This is impossible since $\PP_\r^{(j)}$ and $\PP_\r^{(j')}$ are singular.
Finally, since $\sum_{j=1}^N p_\infty^{(j)}=1$, it implies that $\PP_\r $ a.s. one of the $p_\infty^{(j)}$ is 1 and the others 0. In particular, it implies that for all $j$
$$
\PP_\r(p_\infty^{(j)}=\hbox{0 or 1})=1.
$$


This implies that we have 
$$
\PP_\r^{(j)}=\PP_\r\left[\,\cdot\,\vert\, p_\infty^{(j)}=1\right]
$$
for 
$$
\PP_\r\left[\,\cdot\,\vert\, p_\infty^{(j)}=1\right]=\frac{p_\infty^{(j)}\, \PP_\r}{\PP_\r\left[p_\infty^{(j)}=1\right]}\,,
$$
but $\PP\left[p_\infty^{(j)}=1\right]=p_0^{(j)}$ since $(p_n^{(j)})$ is a martingale.

\smallskip
The conclusion now is a direct consequence of the Central Limit Theorem established for the chain $(X_n)$ but now associated to the family $(A_i^{(j)})_{i=1}^{2d}$
using previous proposition \ref{Restriction}. 

\end{proof}

\bigskip
{\timehuit S. Attal, N. Guillotin-Plantard, C. Sabot}
\vskip -1mm
{\timesept Universit\'e de Lyon}
\vskip -1mm
{\timesept Universit\'e de Lyon 1, C.N.R.S.}
\vskip -1mm
{\timesept Institut Camille Jordan}
\vskip -1mm
{\timesept 21 av Claude Bernard}
\vskip -1mm
{\timesept 69622 Villeurbanne cedex, France}

\bigskip

\end{document}